\documentclass[final]{siamltex}

\usepackage{amssymb}
\usepackage[leqno]{amsmath}
\usepackage{mathrsfs}
\usepackage{stmaryrd}
\usepackage{chemarrow}
\usepackage{algorithm2e}
\usepackage{enumerate}
\usepackage{graphicx}
\usepackage[pdftex,bookmarksnumbered,bookmarksopen,colorlinks,linkcolor=red,anchorcolor=black,citecolor=blue,urlcolor=blue]{hyperref}
\usepackage[all]{xy}
\usepackage{tikz}
\usetikzlibrary{arrows}

\newtheorem{remark}[theorem]{Remark}

\title{Stabilized mixed finite element methods for linear elasticity on simplicial grids in $\mathbb{R}^{\MakeLowercase{n}}$}%
\author{
Long Chen\thanks{Department of Mathematics, University of California at Irvine, Irvine, CA 92697, USA (chenlong@math.uci.edu). L. Chen was supported by  NSF Grant DMS-1418934. This work was finished when L. Chen visited Peking University in the fall of 2015. He would like to thank Peking University for the support and hospitality, as well as for their exciting research atmosphere.}
\and
Jun Hu\thanks{LMAM and School of Mathematical Sciences, Peking University, Beijing 100871, China (hujun@math.pku.edu.cn). The work of this author was supported by  the NSFC Projects 11625101, 11271035,  91430213 and 11421101.}
\and
Xuehai Huang\thanks{Corresponding author. College of Mathematics and Information Science, Wenzhou University, Wenzhou 325035, China (xuehaihuang@gmail.com). The work of this author was supported by the NSFC Projects 11301396 and 11671304, and Zhejiang Provincial
Natural Science Foundation of China Projects LY17A010010, LY15A010015, LY15A010016 and LY14A010020.}
}%

\begin{document}

\maketitle

\begin{abstract} In this paper, we design two classes of stabilized mixed finite element methods for linear elasticity on simplicial grids.
In the first class of elements, we use  $\boldsymbol{H}(\mathbf{div}, \Omega; \mathbb{S})$-$P_k$ and $\boldsymbol{L}^2(\Omega; \mathbb{R}^n)$-$P_{k-1}$ to approximate the stress and displacement spaces, respectively,  for $1\leq k\leq n$, and employ a stabilization  technique in terms  of
the jump of  the discrete displacement over the faces of the triangulation under consideration; in the second class of elements, we use $\boldsymbol{H}_0^1(\Omega; \mathbb{R}^n)$-$P_{k}$ to approximate the displacement space for $1\leq k\leq n$, and adopt the stabilization technique suggested by  Brezzi, Fortin, and  Marini. We establish the discrete inf-sup conditions,  and  consequently present the a priori error analysis for  them.  The main ingredient for the analysis is  two special interpolation operators, which can be constructed using a crucial  $\boldsymbol{H}(\mathbf{div})$ bubble function space of polynomials on each element. The feature of these  methods is the low number of global degrees of freedom in the lowest order case. We present  some numerical results  to demonstrate the theoretical estimates.
\end{abstract}



\section{Introduction}

Assume that $\Omega\subset \mathbb{R}^n$ is a bounded
polytope. Denote by $\mathbb{S}$ the space of all symmetric $n\times n$ tensors. The Hellinger-Reissner mixed formulation of the linear elasticity under the load $\boldsymbol{f}\in\boldsymbol{L}^2(\Omega; \mathbb{R}^n)$ is given as follows:
Find $(\boldsymbol{\sigma}, \boldsymbol{u})\in \boldsymbol{\Sigma}\times \boldsymbol{V}:=\boldsymbol{H}(\mathbf{div}, \Omega; \mathbb{S})\times\boldsymbol{L}^2(\Omega; \mathbb{R}^n)$ such that
\begin{align}
a(\boldsymbol{\sigma},\boldsymbol{\tau})+b(\boldsymbol{\tau}, \boldsymbol{u})&=0 \quad\quad\quad\quad\quad\quad \forall\,\boldsymbol{\tau}\in\boldsymbol{\Sigma}, \label{mixedform1} \\
-b(\boldsymbol{\sigma}, \boldsymbol{v})&=\int_{\Omega}\boldsymbol{f}\cdot\boldsymbol{v}\,{\rm d}x \quad\quad \forall\,\boldsymbol{v}\in \boldsymbol{V}, \label{mixedform2}
\end{align}
where
\[
a(\boldsymbol{\sigma},\boldsymbol{\tau}):=\int_{\Omega}\mathcal{A}\boldsymbol{\sigma}:\boldsymbol{\tau}\,{\rm d}x,\quad
b(\boldsymbol{\tau}, \boldsymbol{v}):= \int_{\Omega}\mathbf{div}\boldsymbol{\tau}\cdot\boldsymbol{v}\,{\rm d}x
\]
with $\mathcal{A}$ being the compliance tensor of fourth order defined
by
\[
\mathcal{A}\boldsymbol{\sigma}:=\frac{1}{2\mu}\left(\boldsymbol{\sigma}
-\frac{\lambda}{n\lambda+2\mu}(\textrm{tr}\boldsymbol{\sigma})\boldsymbol{\delta}\right).
\]
Here $\boldsymbol{\delta}:=(\delta_{ij})_{n\times n}$ is the
Kronecker tensor, tr is the trace operator, and positive constants $\lambda$ and $\mu$ are the
Lam$\acute{e}$ constants.
It is arduous to design $\boldsymbol{H}(\mathbf{div}, \Omega; \mathbb{S})$ conforming finite element with polynomial shape functions due to the symmetry requirement of the stress tensor.
Hence composite elements were one main choice to approximate the stress in the last century (cf. \cite{ArnoldDouglasGupta1984, Veubeke1965, JohnsonMercier1978, WatwoodHartz1968}).

In the early years of this century, Arnold and Winther constructed the first $\boldsymbol{H}(\mathbf{div}, \Omega; \mathbb{S})$ conforming mixed finite element with polynomial shape functions in two dimensions in \cite{ArnoldWinther2002}, which was extended to tetrahedral grids in three dimensions in \cite{AdamsCockburn2005,ArnoldAwanouWinther2008} and simplicial grids in any dimension as a byproduct in \cite{HuZhang2015b}. In those elements,
 the displacement space is approximated by $\boldsymbol{L}^2(\Omega; \mathbb{R}^n)$-$P_{k-1}$ while the stress space is approximated by
 the space of functions in $\boldsymbol{H}(\mathbf{div}, \Omega; \mathbb{S})$-$P_{k+n-1}$
whose divergence is in $\boldsymbol{L}^2(\Omega; \mathbb{R}^n)$-$P_{k-1}$ for $k\geq 2$.
Recently, Hu and Zhang showed that the more compact pair of $\boldsymbol{H}(\mathbf{div}, \Omega; \mathbb{S})$-$P_k$ and $\boldsymbol{L}^2(\Omega; \mathbb{R}^n)$-$P_{k-1}$ spaces is stable on triangular and tetrahedral grids for $k\geq n+1$ with $n=2, 3$ in \cite{HuZhang2015c, HuZhang2015}.
And Hu generalized those stable finite elements to simplicial grids in any dimension for $k\geq n+1$ in \cite{Hu2015a}.
 One key observation there is that the divergence space of the $\boldsymbol{H}(\mathbf{div})$ bubble function space of polynomials  on each element
  is just the orthogonal complement space of the piecewise rigid motion space with respect to the discrete displacement space. Then the discrete inf-sup condition was proved for $k\geq n+1$ through controlling the piecewise rigid motion space
by $\boldsymbol{H}^1(\Omega; \mathbb{S})$-$P_k$ space.
It is, however, troublesome to prove that the pair of $\boldsymbol{H}(\mathbf{div}, \Omega; \mathbb{S})$-$P_k$ and $\boldsymbol{L}^2(\Omega; \mathbb{R}^n)$-$P_{k-1}$ is still stable for $1\leq k\leq n$. For such a reason, Hu and Zhang enriched the $\boldsymbol{H}(\mathbf{div}, \Omega; \mathbb{S})$-$P_k$ space with $\boldsymbol{H}(\mathbf{div}, \Omega; \mathbb{S})$-$P_{n+1}$ face-bubble functions of piecewise polynomials for each $n-1$ dimensional simplex in \cite{HuZhang2015b}. Gong, Wu and Xu constructed two types of interior penalty mixed finite element methods by using nonconforming symmetric stress approximation  in \cite{GongWuXu2015}. The stability of those nonconforming mixed methods is ensured by $\boldsymbol{H}(\mathbf{div})$  nonconforming face-bubble spaces.
An interior penalty mixed finite element method using Crouzeix-Raviart nonconforming linear element to approximate the stress was studied in \cite{CaiYe2005}.
To get rid of the vertex degrees of freedom appeared in the  $\boldsymbol{H}(\mathbf{div}, \Omega; \mathbb{S})$-conforming elements and make the resulting mixed finite element methods hybridizable, nonconforming mixed elements on triangular and tetrahedral grids were developed in \cite{ArnoldWinther2003a,ArnoldAwanouWinther2014,GopalakrishnanGuzman2011, XieXu2011}.
On rectangular grids, we refer to \cite{ArnoldAwanou2005,Hu2015,HuManZhang2014,Awanou2012,ChenWang2011} for symmetric conforming mixed finite elements and \cite{HuShi2007a,ManHuShi2009,Yi2005,Yi2006} for symmetric nonconforming mixed finite elements.
For keeping the symmetry of the discrete stress space and relaxing the continuity across the interior faces of the triangulation,  many discontinuous Galerkin methods were proposed in \cite{Bustinza2006,CockburnSchotzauWang2006,ChenHuangHuangXu2010, HuangHuang2016,Huang2013a, HuangHuang2013a}, hybridizable discontinuous Galerkin methods in \cite{QiuShenShi2013, HarderMadureiraValentin2016}, weak Galerkin methods in \cite{WangWangWangZhang2016, ChenXie2016}, hybrid high-order method in \cite{DiPietroErn2015}. For the weakly symmetric mixed finite element methods for linear elasticity, we refer to \cite{AmaraThomas1979,ArnoldBrezziDouglas1984,Stenberg1988,Morley1989, ArnoldFalkWinther2007,BoffiBrezziFortin2009,QiuDemkowicz2009,CockburnGopalakrishnanGuzman2010,Guzman2010,QiuDemkowicz2011,GopalakrishnanGuzman2012,CockburnShi2013,ArnoldLee2014}.

In this paper, we intend to design stable mixed finite element methods for the linear elasticity using as few global degrees of freedom as possible.
To this end, two  classes of stabilized mixed finite element methods on simplicial grids in any dimension are proposed. In the first one, we use the pair of $\boldsymbol{H}(\mathbf{div}, \Omega; \mathbb{S})$-$P_k$ and $\boldsymbol{L}^2(\Omega; \mathbb{R}^n)$-$P_{k-1}$ constructed in \cite{Hu2015a} to approximate the stress and displacement for $1\leq k\leq n$. To simplify the notation, we shall use superscript $(\cdot)^{\rm div}$ for $\boldsymbol{H}(\mathbf{div}, \Omega; \mathbb{S})$ conforming elements, $(\cdot)^{-1}$ for discontinuous elements, and $(\cdot)^{0}$ for $\boldsymbol{H}^1(\Omega; \mathbb{R}^n)$ or $\boldsymbol{H}^1(\Omega; \mathbb{S})$ continuous elements. Instead of enriching $P_k^{\rm div}$ elements with face bubble functions of piecewise polynomials as in \cite{HuZhang2015b, GongWuXu2015}, we include a jump stabilization term into the  Hellinger-Reissner mixed formulation to make the discrete method stable, inspired by the discontinuous Galerkin methods constructed in \cite{ChenHuangHuangXu2010} for the linear elasticity problem.
The discrete inf-sup condition in a compact form is established with the help of a partial inf-sup condition \eqref{rmperp} estabilished in \cite{Hu2015a, HuZhang2015c, HuZhang2015} and a well-tailored interpolation operator for the stress.
Then we show the a priori error analysis for the resulting stabilized $P_k^{\rm div} - P_{k-1}^{-1}$ elements.

In the second class of stabilized mixed finite element methods, we adopt the stabilization technique suggested in \cite{BrezziFortinMarini1993} and use $\boldsymbol{H}_0^1(\Omega; \mathbb{R}^n)$-$P_{k}$ to approximate the displacement space. The merit of this stabilization technique is that the coercivity condition for the bilinear form related to the stress holds automatically, thus we only need to focus on the discrete inf-sup condition.
To recover the inf-sup condition, we first employ $\boldsymbol{H}^1(\Omega; \mathbb{S})$-$P_{k}$ enriched with $(k+1)$-th order $\boldsymbol{H}(\mathbf{div})$ bubble  function space of polynomials on each element to approximate the stress space. The discrete inf-sup condition is established by using another special interpolation operator for the stress.
The a priori error estimate for the $(P_k^0 + B_{k+1}^{\rm div}) - P_k^0$ is then derived by the standard theory of mixed finite element methods.
The rate of convergence for the displacement in $\boldsymbol{L}^2(\Omega; \mathbb{R}^n)$ norm is, however, suboptimal due to the coupling of stress error measured in $\boldsymbol{H}(\mathbf{div})$-norm.
To remedy this, we use $\boldsymbol{H}(\mathbf{div}, \Omega; \mathbb{S})$-$P_{k+1}$ to approximate the stress space instead. The resulting stable finite element pair is the Hood-Taylor type $P_{k+1}^{\rm div} - P_{k}^0$.
It was mentioned in \cite{BrezziFortinMarini1993} that it is not known if the Hood-Taylor element in \cite{TaylorHood1973,Boffi1994,Boffi1997} is stable for the linear elasticity. We solve this problem by enriching the Hood-Taylor element space $P_{k+1}^0$ for the stress with the same degree of $\boldsymbol{H}(\mathbf{div})$ bubble function space of polynomials on each element.

Note that the key component in constructing the previous two interpolation operators is the $\boldsymbol{H}(\mathbf{div})$ bubble function space of polynomials on each element.  We would like to mention that for the stress we obtain the optimal error estimates  in $\boldsymbol{H}(\mathbf{div}, \Omega; \mathbb{S})$ norm, whereas these error estimates in $\boldsymbol{L}^2(\Omega; \mathbb{S})$ norm are suboptimal. To the best of our knowledge, the global degrees of freedom of our  methods for the lowest order case $k=1$ are fewer than those of any existing mixed-type symmetric finite element method for the linear elasticity in the literature. To be specific, the global degrees of freedom for the stress and the displacement for the stabilized mixed finite element methods~\eqref{stabmfem1}-\eqref{stabmfem2}, \eqref{continuousmfem1}-\eqref{continuousmfem2} and \eqref{continuousfullmfem1}-\eqref{continuousfullmfem2} with $k=1$ are $\frac{n(n+1)}{2}|\mathcal{V}|+n|\mathcal{T}|$, $\frac{n(n+1)}{2}(|\mathcal{V}|+|\mathcal{T}|)+n|\mathcal{V}|$ and $\frac{n(n+1)}{2}|\mathcal{V}|+\frac{(n-1)(n+2)}{2}|\mathcal{E}|+\frac{n(n+1)}{2}|\mathcal{T}|+n|\mathcal{V}|$, respectively. Here $|\mathcal{V}|, |\mathcal{E}|, |\mathcal{T}|$ are the numbers of vertices, edges and elements of the triangulation.

When adopting the same degree of polynomial spaces for displacement,
the Taylor-Hood type elements $P_{k+1}^{\rm div} - P_{k}^0$ and the stabilized elements $P_{k+1}^{\rm div} - P_{k}^{-1}$ share the same convergence rate, which is one order higher than the stabilized elements $(P_k^0 + B_{k+1}^{\rm div}) - P_k^0$.
It is worth mentioning that to keep the same convergence rate, the stabilized elements $P_{k+1}^{\rm div} - P_{k}^{-1}$ need larger global degrees of freedom than the Taylor-Hood type elements $P_{k+1}^{\rm div} - P_{k}^0$.

%

The rest of this article is organized as follows. We present some notations and definitions in Section 2 for later uses.
In Section 3, a stabilized mixed finite element method with discontinuous displacement for the linear elasticity is designed and analyzed. Then we propose a second class of stabilized mixed finite element methods with continuous displacement for the linear elasticity in Section 4.
In Section 5, some numerical experiments are given to demonstrate the theoretical results.



\section{Preliminaries}
Given a bounded domain $G\subset\mathbb{R}^{n}$ and a
non-negative integer $m$, let $H^m(G)$ be the usual Sobolev space of functions
on $G$, and $\boldsymbol{H}^m(G; \mathbb{X})$ be the usual Sobolev space of functions taking values in the finite-dimensional vector space $\mathbb{X}$ for $\mathbb{X}$ being $\mathbb{S}$ or $\mathbb{R}^n$. The corresponding norm and semi-norm are denoted respectively by
$\Vert\cdot\Vert_{m,G}$ and $|\cdot|_{m,G}$.  If $G$ is $\Omega$, we abbreviate
them by $\Vert\cdot\Vert_{m}$ and $|\cdot|_{m}$,
respectively. Let $\boldsymbol{H}_0^m(G; \mathbb{R}^n)$ be the closure of $\boldsymbol{C}_{0}^{\infty}(G; \mathbb{R}^n)$ with
respect to the norm $\Vert\cdot\Vert_{m,G}$.
Denote by $\boldsymbol{H}(\mathbf{div}, G; \mathbb{S})$ the Sobolev space of square-integrable symmetric tensor fields with square-integrable divergence.
For any $\boldsymbol{\tau}\in\boldsymbol{H}(\mathbf{div}, \Omega; \mathbb{S})$, we equip the following norm
\[
\|\boldsymbol{\tau}\|_{\boldsymbol{H}(\mathbf{div},\mathcal{A})}^2:=a(\boldsymbol{\tau},\boldsymbol{\tau})+\|\mathbf{div}\boldsymbol{\tau}\|_0^2.
\]
When $\boldsymbol{\tau}\in\boldsymbol{H}(\mathbf{div}, \Omega; \mathbb{S})$ satisfying $\int_{\Omega}\textrm{tr}\boldsymbol{\tau}\,{\rm d}x=0$, it follows from
Proposition~9.1.1 in \cite{BoffiBrezziFortin2013} that there exists a constant $C>0$ such that
\[
\|\boldsymbol{\tau}\|_{0}\leq C\|\boldsymbol{\tau}\|_{\boldsymbol{H}(\mathbf{div},\mathcal{A})},
\]
which means $\|\boldsymbol{\tau}\|_{\boldsymbol{H}(\mathbf{div},\mathcal{A})}$ and $\|\boldsymbol{\tau}\|_{\boldsymbol{H}(\mathbf{div})}$ are equivalent uniformly with respect to the
Lam$\acute{e}$ constant $\lambda$. Hence the norm $\|\cdot\|_{\boldsymbol{H}(\mathbf{div},\mathcal{A})}$ presented in all of the estimates in this paper can be
replaced by the norm $\|\cdot\|_{\boldsymbol{H}(\mathbf{div})}$.

Suppose the domain $\Omega$ is subdivided by a family of shape regular simplicial grids $\mathcal {T}_h$  (cf.\ \cite{BrennerScott2008,Ciarlet1978}) with $h:=\max\limits_{K\in \mathcal {T}_h}h_K$
and $h_K:=\mbox{diam}(K)$.
Let $\mathcal{F}_h$ be the union of all $n-1$ dimensional faces
of $\mathcal {T}_h$. For any $F\in\mathcal{F}_h$,
denote by $h_F$ its diameter.
Let $P_m(G)$ stand for the set of all
polynomials in $G$ with the total degree no more than $m$, and $\boldsymbol{P}_m(G; \mathbb{X})$ denote the tensor or vector version of $P_m(G)$ for $\mathbb{X}$ being $\mathbb{S}$ or $\mathbb{R}^n$, respectively.
Throughout this paper, we also use
``$\lesssim\cdots $" to mean that ``$\leq C\cdots$", where
$C$ is a generic positive constant independent of $h$ and the
Lam$\acute{e}$ constant $\lambda$,
which may take different values at different appearances.

Consider two adjacent simplices $K^+$ and $K^-$ sharing an interior face $F$.
Denote by $\boldsymbol{\nu}^+$ and $\boldsymbol{\nu}^-$ the unit outward normals
to the common face $F$ of the simplices $K^+$ and $K^-$, respectively.  For a
vector-valued function $\boldsymbol{w}$, write $\boldsymbol{w}^+:=\boldsymbol{w}|_{K^+}$ and $\boldsymbol{w}^-:=\boldsymbol{w}|_{K^-}$.
Then define a matrix-valued jump as
\[
\llbracket\boldsymbol{w}\rrbracket:=\frac{1}{2}\left(\boldsymbol{w}^+(\boldsymbol{\nu}^+)^T+\boldsymbol{\nu}^+(\boldsymbol{w}^+)^T+\boldsymbol{w}^-(\boldsymbol{\nu}^-)^T+\boldsymbol{\nu}^-(\boldsymbol{w}^-)^T\right).
\]
On a face $F$ lying on the boundary $\partial\Omega$, the above term is defined by
\[
\llbracket\boldsymbol{w}\rrbracket:=\frac{1}{2}\left(\boldsymbol{w}\boldsymbol{\nu}^T+\boldsymbol{\nu}\boldsymbol{w}^T\right).
\]

For each $K\in\mathcal{T}_h$, define an $\boldsymbol{H}(\mathbf{div}, K; \mathbb{S})$ bubble function space of polynomials of degree $k$ as
\[
\boldsymbol{B}_{K,k}:=\left\{\boldsymbol{\tau}\in \boldsymbol{P}_{k}(K; \mathbb{S}): \boldsymbol{\tau}\boldsymbol{\nu}|_{\partial K}=\boldsymbol{0}\right\}.
\]
It is easy to check that $\boldsymbol{B}_{K,1}$ is merely the zero space.
Denote the vertices of simplex $K$ by $\boldsymbol{x}_0, \cdots, \boldsymbol{x}_n$.
For any edge $\boldsymbol{x}_i\boldsymbol{x}_j$($i\neq j$) of element $K$, let $\boldsymbol{t}_{i,j}$ be the associated tangent vectors and
\[
\boldsymbol{T}_{i,j}:=\boldsymbol{t}_{i,j}\boldsymbol{t}_{i,j}^T,\quad 0\leq i<j\leq n.
\]
It follows from \cite{Hu2015a} that, for $k\geq 2$,
\[
\boldsymbol{B}_{K,k}=\sum_{0\leq i<j\leq n}\lambda_i\lambda_jP_{k-2}(K)\boldsymbol{T}_{i,j},
\]
where $\lambda_i$ is the associated barycentric coordinates corresponding to $\boldsymbol{x}_i$ for $i=0,\cdots,n$.
Some global finite element spaces are given by
\begin{align*}
&\boldsymbol{B}_{k,h}:=\left\{\boldsymbol{\tau}\in\boldsymbol{H}(\mathbf{div}, \Omega; \mathbb{S}):
\boldsymbol{\tau}|_K\in \boldsymbol{B}_{K,k} \quad \forall\,K\in\mathcal{T}_h \right\}, \\
&\widetilde{\boldsymbol{\Sigma}}_{k,h}:=\left\{\boldsymbol{\tau}\in\boldsymbol{H}^1(\Omega; \mathbb{S}):
\boldsymbol{\tau}|_K\in \boldsymbol{P}_{k}(K; \mathbb{S}) \quad \forall\,K\in\mathcal{T}_h \right\}, \\
&\boldsymbol{\Sigma}_{k,h}:=\widetilde{\boldsymbol{\Sigma}}_{k,h} + \boldsymbol{B}_{k,h},\label{spaceprop}\\
&\boldsymbol{V}_{k-1,h}:=\left\{\boldsymbol{v}\in \boldsymbol{L}^2(\Omega; \mathbb{R}^n): \boldsymbol{v}|_K\in \boldsymbol{P}_{k-1}(K; \mathbb{R}^n)\quad \forall\,K\in\mathcal
{T}_h\right\},
\end{align*}
with integer $ k\geq 1$. It follows from \cite{Hu2015a, HuZhang2015c, HuZhang2015} that
\begin{equation}\label{rmperp}
\boldsymbol{R}^{\perp}(K)= \mathbf{div}\boldsymbol{B}_{K,k} \quad \forall~K\in\mathcal{T}_h,
\end{equation}
where the local rigid motion space and its orthogonal complement space with respect to $\boldsymbol{P}_{k-1}(K; \mathbb{R}^n)$ on each simplex $K\in\mathcal{T}_h$ are defined as (cf. \cite{Hu2015a})
\begin{align*}
\boldsymbol{R}(K):= & \left\{\boldsymbol{v}\in \boldsymbol{H}^1(K; \mathbb{R}^n): \boldsymbol{\varepsilon}(\boldsymbol{v})=\boldsymbol{0}\right\}, \\
\boldsymbol{R}^{\perp}(K):= & \left\{\boldsymbol{v}\in \boldsymbol{P}_{k-1}(K; \mathbb{R}^n): \int_{K}\boldsymbol{v}\cdot\boldsymbol{w}\,{\rm d}x=0 \quad\forall~\boldsymbol{w}\in\boldsymbol{R}(K)\right\},
\end{align*}
with
$\boldsymbol{\varepsilon}(\boldsymbol{v}):=\left(\boldsymbol{\nabla}\boldsymbol{v}+(\boldsymbol{\nabla}\boldsymbol{v})^T\right)/2$ being the linearized strain tensor.

To introduce an elementwise $\boldsymbol{H}(\mathbf{div})$ bubble function interpolation operator, we first present the degrees of freedom for $\boldsymbol{\Sigma}_{k,h}$ which are slightly different from those given in \cite{Hu2015a}. For the ease of notation, we understand $P_{k} = \emptyset$ for negative integers $k$.
\begin{lemma}\label{lem:dofs}
A matrix field $\boldsymbol{\tau}\in\boldsymbol{P}_{k}(K; \mathbb{S})$ can be uniquely determined by the degrees of freedom from
\begin{enumerate}[(1)]
\item For each $\ell$ dimensional simplex $\Delta_{\ell}$ of $K$, $0\leq \ell\leq n-1$, with $\ell$ linearly independent
tangential vectors $\boldsymbol{t}_1, \cdots, \boldsymbol{t}_{\ell}$, and $n-\ell$ linearly independent normal vectors $\boldsymbol{\nu}_1, \cdots, \boldsymbol{\nu}_{n-\ell}$,
the mean moments of degree at most $k-\ell-1$ over $\Delta_{\ell}$, of $\boldsymbol{t}_l^T\boldsymbol{\tau}\boldsymbol{\nu}_i$, $\boldsymbol{\nu}_i^T\boldsymbol{\tau}\boldsymbol{\nu}_j$, $l=1,\cdots, \ell$, $i,j=1,\cdots, n-\ell$;
\item the values $\int_K \boldsymbol{\tau}:\boldsymbol{\varsigma}dx$ for any $\boldsymbol{\varsigma}\in\boldsymbol{P}_{k-2}(K; \mathbb{S})$.
\end{enumerate}
\end{lemma}
\begin{proof}
This lemma can be proved by applying the arguments used in Theorems 2.1-2.2 in \cite{Hu2015a}.
\end{proof}

It is easy to see that we have the same first set of degrees of freedom as those in \cite{Hu2015a},  whereas the second set of degrees of freedom is different.

Now we present an elementwise $\boldsymbol{H}(\mathbf{div})$ bubble function interpolation operator.
Given $\boldsymbol{\tau}\in\boldsymbol{L}^{2}(\Omega; \mathbb{S})$, define $\boldsymbol{I}_{k,h}^b\boldsymbol{\tau}\in\boldsymbol{\Sigma}_{k,h}$ as follows:  on each simplex $K\in\mathcal{T}_h$,
\begin{itemize}
\item for any degree of freedom $D$ in the first set of degrees of freedom in Lemma~\ref{lem:dofs},
\[
D(\boldsymbol{I}_{k,h}^b\boldsymbol{\tau})=0,
\]
\item for any $\boldsymbol{\varsigma}\in\boldsymbol{P}_{k-2}(K; \mathbb{S})$,

\hfill \makebox[0pt][r]{%
\begin{minipage}[b]{\textwidth}
\begin{equation}\label{eq:Ihbdef}
\int_K \boldsymbol{I}_{k,h}^b\boldsymbol{\tau}:\boldsymbol{\varsigma}dx =\int_K \boldsymbol{\tau}:\boldsymbol{\varsigma}dx.
\end{equation}
\end{minipage}}
\end{itemize}
Since the first set of degrees of freedom in Lemma~\ref{lem:dofs} completely determines $\boldsymbol{\tau}\boldsymbol{\nu}$ on $\partial K$ for any $\boldsymbol{\tau}\in\boldsymbol{P}_{k}(K; \mathbb{S})$ (cf. \cite[Theorem~2.1]{Hu2015a}), thus $\boldsymbol{I}_{k,h}^b\boldsymbol{\tau}\in\boldsymbol{B}_{k,h}$. Applying scaling argument, we have for any $\boldsymbol{\tau}\in\boldsymbol{L}^{2}(\Omega; \mathbb{S})$
\begin{equation}\label{eq:Ihbbound}
\|\boldsymbol{I}_{k,h}^b\boldsymbol{\tau}\|_{0,K}\lesssim\|\boldsymbol{\tau}\|_{0,K} \quad \forall~K\in\mathcal{T}_h.
\end{equation}

\section{A stabilized mixed finite element method  with discontinuous displacement}

In this section, we devise a stabilized mixed finite element method for linear elasticity.
The pair of $\boldsymbol{H}(\mathbf{div}, \Omega; \mathbb{S})$-$P_k$ and $\boldsymbol{L}^2(\Omega; \mathbb{R}^n)$-$P_{k-1}$ is shown to be stable for $k\geq n+1$ in \cite{Hu2015a}. Here we consider the range $1\leq k\leq n$.

With previous preparation,
a stabilized mixed finite element method for linear elasticity is defined as follows:
Find $(\boldsymbol{\sigma}_h, \boldsymbol{u}_h)\in \boldsymbol{\Sigma}_{k,h}\times \boldsymbol{V}_{k-1,h}$ such that
\begin{align}
a(\boldsymbol{\sigma}_h,\boldsymbol{\tau}_h)+b(\boldsymbol{\tau}_h, \boldsymbol{u}_h)&=0 \quad\quad\quad\quad\quad\quad\quad\;\; \forall\,\boldsymbol{\tau}_h\in\boldsymbol{\Sigma}_{k,h}, \label{stabmfem1} \\
-b(\boldsymbol{\sigma}_h, \boldsymbol{v}_h) + c(\boldsymbol{u}_h, \boldsymbol{v}_h)&=\int_{\Omega}\boldsymbol{f}\cdot\boldsymbol{v}_h\,{\rm d}x \quad\quad\quad \forall\,\boldsymbol{v}_h\in \boldsymbol{V}_{k-1,h}, \label{stabmfem2}
\end{align}
where the jump stabilization term for the displacement
$
c(\boldsymbol{u}_h,\boldsymbol{v}_h):=\sum\limits_{F\in\mathcal{F}_h}h_{F}\int_{F}\llbracket\boldsymbol{u}_h\rrbracket:\llbracket\boldsymbol{v}_h\rrbracket\,{\rm d}s
$. With this jump stabilization term, a jump seminorm for $\boldsymbol{V}_{k-1,h}+ \boldsymbol{H}^{1}(\Omega; \mathbb{R}^n)$ is defined as
\[
\|\boldsymbol{v}_h\|_c^2:=c(\boldsymbol{v}_h,\boldsymbol{v}_h) \quad\quad \forall~\boldsymbol{v}_h\in\boldsymbol{V}_{k-1,h}+ \boldsymbol{H}^{1}(\Omega; \mathbb{R}^n).
\]
We also define the following two norms
\begin{align*}
\|\boldsymbol{\tau}\|_a^2&:=a(\boldsymbol{\tau}, \boldsymbol{\tau}) \quad\quad\quad\quad\quad\; \forall~\boldsymbol{\tau}\in\boldsymbol{L}^2(\Omega; \mathbb{S}), \\
\|\boldsymbol{v}_h\|_{0,c}^2&:=\|\boldsymbol{v}_h\|_{0}^2+\|\boldsymbol{v}_h\|_{c}^2 \quad\quad \forall~\boldsymbol{v}_h\in\boldsymbol{V}_{k-1,h}+ \boldsymbol{H}^{1}(\Omega; \mathbb{R}^n).
\end{align*}



Let $\boldsymbol{Q}_h$ be the $L^2$ orthogonal projection from $\boldsymbol{L}^2(\Omega; \mathbb{R}^n)$ onto $\boldsymbol{V}_{k-1,h}$. It holds the following error estimate (cf. \cite{Ciarlet1978, BrennerScott2008})
\begin{equation}\label{eq:L2projErrorEstimate}
\|\boldsymbol{v}-\boldsymbol{Q}_h\boldsymbol{v}\|_{0,K}+h_K^{1/2}\|\boldsymbol{v}-\boldsymbol{Q}_h\boldsymbol{v}\|_{0,\partial K}\lesssim h_K^{\min\{k,m\}}|\boldsymbol{v}|_{m,K} \quad \forall~\boldsymbol{v}\in \boldsymbol{H}^{m}(\Omega; \mathbb{R}^n)
\end{equation}
with integer $m\geq1$.
Let $\boldsymbol{I}_h^{SZ}$ be a tensorial or vectorial Scott-Zhang interpolation operator designed in \cite{ScottZhang1990}, which possesses the following error estimate
\begin{equation}\label{eq:IhSZErrorEstimate}
\sum_{K\in\mathcal{T}_h}h_K^{-2}\|\boldsymbol{\tau}-\boldsymbol{I}_h^{SZ}\boldsymbol{\tau}\|_{0,K}^2+|\boldsymbol{\tau}-\boldsymbol{I}_h^{SZ}\boldsymbol{\tau}|_{1}^2\lesssim h^{2\min\{k,m-1\}}\|\boldsymbol{\tau}\|_{m}^2
\end{equation}
for any $\boldsymbol{\tau}\in \boldsymbol{H}^{m}(\Omega; \mathbb{S})$ with integer $m\geq1$. Then for each $\boldsymbol{\tau}\in \boldsymbol{H}^{1}(\Omega; \mathbb{S})$,
define  $\boldsymbol{I}_h\boldsymbol{\tau}:=\boldsymbol{I}_h^{SZ}\boldsymbol{\tau}+\boldsymbol{I}_{k,h}^b(\boldsymbol{\tau}-\boldsymbol{I}_h^{SZ}\boldsymbol{\tau})$.
Apparently we have $\boldsymbol{I}_h\boldsymbol{\tau}\in\boldsymbol{\Sigma}_{k,h}$. And
it follows from \eqref{eq:Ihbdef}
\begin{equation}\label{eq:Ihprop}
\int_K (\boldsymbol{I}_h\boldsymbol{\tau}-\boldsymbol{\tau}):\boldsymbol{\varsigma}dx =0\quad \forall~\boldsymbol{\varsigma}\in\boldsymbol{P}_{k-2}(K; \mathbb{S}) \textrm{ and } K\in\mathcal{T}_h.
\end{equation}

\begin{lemma}\label{lm:I1}
Let integers $m,k \geq1$.
We have for any $\boldsymbol{\tau}\in \boldsymbol{H}^{m}(\Omega; \mathbb{S})$
\begin{equation}\label{eq:IhErrorEstimate}
\sum_{K\in\mathcal{T}_h}h_K^{-2}\left(\|\boldsymbol{\tau}-\boldsymbol{I}_h\boldsymbol{\tau}\|_{0,K}^2+h_K\|\boldsymbol{\tau}-\boldsymbol{I}_h\boldsymbol{\tau}\|_{0,\partial K}^2\right)\lesssim h^{2\min\{k,m-1\}}\|\boldsymbol{\tau}\|_{m}^2.
\end{equation}
\end{lemma}
\begin{proof}
According to the triangle inequality and \eqref{eq:Ihbbound}, it holds
\[
\|\boldsymbol{\tau}-\boldsymbol{I}_h\boldsymbol{\tau}\|_{0,K}\leq \|\boldsymbol{\tau}-\boldsymbol{I}_h^{SZ}\boldsymbol{\tau}\|_{0,K}+\|\boldsymbol{I}_{k,h}^b(\boldsymbol{\tau}-\boldsymbol{I}_h^{SZ}\boldsymbol{\tau})\|_{0,K}\lesssim \|\boldsymbol{\tau}-\boldsymbol{I}_h^{SZ}\boldsymbol{\tau}\|_{0,K}.
\]
Analogously, we obtain from  the triangle inequality, the inverse inequality, the trace inequality and \eqref{eq:Ihbbound}
\begin{align*}
h_K^{1/2}\|\boldsymbol{\tau}-\boldsymbol{I}_h\boldsymbol{\tau}\|_{0,\partial K}\leq & h_K^{1/2}\|\boldsymbol{\tau}-\boldsymbol{I}_h^{SZ}\boldsymbol{\tau}\|_{0,\partial K} + h_K^{1/2}\|\boldsymbol{I}_{k,h}^b(\boldsymbol{\tau}-\boldsymbol{I}_h^{SZ}\boldsymbol{\tau})\|_{0,\partial K} \\
\lesssim & h_K^{1/2}\|\boldsymbol{\tau}-\boldsymbol{I}_h^{SZ}\boldsymbol{\tau}\|_{0,\partial K} + \|\boldsymbol{I}_{k,h}^b(\boldsymbol{\tau}-\boldsymbol{I}_h^{SZ}\boldsymbol{\tau})\|_{0,K} \\
\lesssim & \|\boldsymbol{\tau}-\boldsymbol{I}_h^{SZ}\boldsymbol{\tau}\|_{0,K} + h_K|\boldsymbol{\tau}-\boldsymbol{I}_h^{SZ}\boldsymbol{\tau}|_{1,K}.
\end{align*}
Thus the combination of the last two inequality and \eqref{eq:IhSZErrorEstimate} implies \eqref{eq:IhErrorEstimate}.
\end{proof}

To derive a discrete inf-sup condition for the stabilized mixed finite element method \eqref{stabmfem1}-\eqref{stabmfem2}, we rewrite it in a compact way:
Find $(\boldsymbol{\sigma}_h, \boldsymbol{u}_h)\in \boldsymbol{\Sigma}_{k,h}\times \boldsymbol{V}_{k-1,h}$ such that
\begin{equation}\label{stabmfemcompact}
\mathcal B(\boldsymbol{\sigma}_h, \boldsymbol{u}_h; \boldsymbol{\tau}_h, \boldsymbol{v}_h)=\int_{\Omega}\boldsymbol{f}\cdot\boldsymbol{v}_h\,{\rm d}x \quad \forall~(\boldsymbol{\tau}_h, \boldsymbol{v}_h)\in \boldsymbol{\Sigma}_{k,h}\times \boldsymbol{V}_{k-1,h},
\end{equation}
where
\[
\mathcal B(\boldsymbol{\sigma}_h, \boldsymbol{u}_h; \boldsymbol{\tau}_h, \boldsymbol{v}_h):=a(\boldsymbol{\sigma}_h,\boldsymbol{\tau}_h)+b(\boldsymbol{\tau}_h, \boldsymbol{u}_h)-b(\boldsymbol{\sigma}_h, \boldsymbol{v}_h) + c(\boldsymbol{u}_h, \boldsymbol{v}_h).
\]
Similarly, problem~\eqref{mixedform1}-\eqref{mixedform2} can be rewritten as
\begin{equation}\label{mixedformcompact}
\mathcal B(\boldsymbol{\sigma}, \boldsymbol{u}; \boldsymbol{\tau}, \boldsymbol{v})=\int_{\Omega}\boldsymbol{f}\cdot\boldsymbol{v}\,{\rm d}x \quad \forall~(\boldsymbol{\tau}, \boldsymbol{v})\in \boldsymbol{\Sigma}\times \boldsymbol{V}.
\end{equation}

Obviously the bilinear form $\mathcal B$ is continuous with respect to the norm $\|\cdot \|_{\boldsymbol{H}(\mathbf{div},\mathcal{A})} + \|\cdot\|_{0,c}$.
Now we present the following inf-sup condition for \eqref{stabmfemcompact}.
\begin{lemma}\label{infsup}
For any $(\widetilde{\boldsymbol{\sigma}}_h, \widetilde{\boldsymbol{u}}_h)\in \boldsymbol{\Sigma}_{k,h}\times \boldsymbol{V}_{k-1,h}$, it follows
\begin{equation}\label{eq:infsup}
\|\widetilde{\boldsymbol{\sigma}}_h\|_{\boldsymbol{H}(\mathbf{div},\mathcal{A})} + \|\widetilde{\boldsymbol{u}}_h\|_{0,c}\lesssim \sup_{(\boldsymbol{\tau}_h, \boldsymbol{v}_h)\in \boldsymbol{\Sigma}_{k,h}\times \boldsymbol{V}_{k-1,h}}\frac{\mathcal B(\widetilde{\boldsymbol{\sigma}}_h, \widetilde{\boldsymbol{u}}_h; \boldsymbol{\tau}_h, \boldsymbol{v}_h)}{\|\boldsymbol{\tau}_h\|_{\boldsymbol{H}(\mathbf{div},\mathcal{A})} + \|\boldsymbol{v}_h\|_{0,c}}.
\end{equation}
\end{lemma}
\begin{proof}
It is sufficient to prove that: for a given pair $(\widetilde{\boldsymbol{\sigma}}_h, \widetilde{\boldsymbol{u}}_h)\in \boldsymbol{\Sigma}_{k,h}\times \boldsymbol{V}_{k-1,h}$,
there exists $(\boldsymbol{\tau}_h, \boldsymbol{v}_h)\in \boldsymbol{\Sigma}_{k,h}\times \boldsymbol{V}_{k-1,h}$ such that
\begin{align}
\label{eq:thm1estimate1}
\|\widetilde{\boldsymbol{\sigma}}_h\|_{\boldsymbol{H}(\mathbf{div},\mathcal{A})}^2+\|\widetilde{\boldsymbol{u}}_h\|_{0,c}^2 &\lesssim B(\widetilde{\boldsymbol{\sigma}}_h, \widetilde{\boldsymbol{u}}_h; \boldsymbol{\tau}_h, \boldsymbol{v}_h), \quad \text{and }\\
\label{eq:thm1estimate2}
\|\boldsymbol{\tau}_h\|_{\boldsymbol{H}(\mathbf{div},\mathcal{A})}+\|\boldsymbol{v}_h\|_{0,c}& \lesssim  \|\widetilde{\boldsymbol{\sigma}}_h\|_{\boldsymbol{H}(\mathbf{div},\mathcal{A})}+\|\widetilde{\boldsymbol{u}}_h\|_{0,c}.
\end{align}

Let $\widetilde{\boldsymbol{u}}_h^{\perp}\in\boldsymbol{L}^2(\Omega; \mathbb{R}^n)$ such that $\widetilde{\boldsymbol{u}}_h^{\perp}|_K$ is the $L^2$-projection of $\widetilde{\boldsymbol{u}}_h|_K$ onto $\boldsymbol{R}^{\perp}(K)$ for each $K\in\mathcal{T}_h$.
Since \eqref{rmperp}, there exists $\boldsymbol{\tau}_1\in\boldsymbol{B}_{k,h}$ such that (cf. \cite[Lemma~3.3]{HuZhang2015} and \cite[Lemma~3.1]{Hu2015a})
\begin{equation}\label{eq:tau1}
\mathbf{div}\boldsymbol{\tau}_1=\widetilde{\boldsymbol{u}}_h^{\perp}, \quad \|\boldsymbol{\tau}_1\|_{H(\mathbf{div}, K)}\lesssim \|\widetilde{\boldsymbol{u}}_h^{\perp}\|_{0,K}.
\end{equation}
By the definition of $\boldsymbol{\tau}_1$, it holds
\[
\mathcal B(\widetilde{\boldsymbol{\sigma}}_h, \widetilde{\boldsymbol{u}}_h; \boldsymbol{\tau}_1, \boldsymbol{0})=a(\widetilde{\boldsymbol{\sigma}}_h, \boldsymbol{\tau}_1)+b(\boldsymbol{\tau}_1, \widetilde{\boldsymbol{u}}_h)=a(\widetilde{\boldsymbol{\sigma}}_h, \boldsymbol{\tau}_1)+(\widetilde{\boldsymbol{u}}_h^{\perp}, \widetilde{\boldsymbol{u}}_h)=a(\widetilde{\boldsymbol{\sigma}}_h, \boldsymbol{\tau}_1)+\|\widetilde{\boldsymbol{u}}_h^{\perp}\|_0^2.
\]
Thus by \eqref{eq:tau1}, there exists a constant $C_1>0$ such that
\begin{align}
\mathcal B(\widetilde{\boldsymbol{\sigma}}_h, \widetilde{\boldsymbol{u}}_h; \boldsymbol{\tau}_1, 0)\geq & -\|\widetilde{\boldsymbol{\sigma}}_h\|_a\|\boldsymbol{\tau}_1\|_a+\|\widetilde{\boldsymbol{u}}_h^{\perp}\|_0^2\geq -C_1\|\widetilde{\boldsymbol{\sigma}}_h\|_a\|\widetilde{\boldsymbol{u}}_h^{\perp}\|_0+\|\widetilde{\boldsymbol{u}}_h^{\perp}\|_0^2 \notag\\
\geq &- \frac{C_1^2}{2}\|\widetilde{\boldsymbol{\sigma}}_h\|_a^2+ \frac{1}{2}\|\widetilde{\boldsymbol{u}}_h^{\perp}\|_0^2. \label{eq:tau22}
\end{align}
On the other hand, there exists $\boldsymbol{\tau}_2\in\boldsymbol{H}_{0}^1(\Omega;\mathbb{S})$ such that (cf. \cite{ArnoldWinther2002, ArnoldAwanouWinther2008})
\begin{equation}\label{eq:tau4}
\mathbf{div}\boldsymbol{\tau}_2=\widetilde{\boldsymbol{u}}_h-\widetilde{\boldsymbol{u}}_h^{\perp}, \quad \text{and }\; \|\boldsymbol{\tau}_2\|_{1}\lesssim \|\widetilde{\boldsymbol{u}}_h-\widetilde{\boldsymbol{u}}_h^{\perp}\|_{0}.
\end{equation}
Thanks to \eqref{eq:Ihprop}, we have from integration by parts
\begin{align*}
b(\boldsymbol{I}_h\boldsymbol{\tau}_2, \widetilde{\boldsymbol{u}}_h)=&b(\boldsymbol{I}_h\boldsymbol{\tau}_2-\boldsymbol{\tau}_2, \widetilde{\boldsymbol{u}}_h)+b(\boldsymbol{\tau}_2, \widetilde{\boldsymbol{u}}_h) \\
=&\sum_{F\in\mathcal{F}_h}\int_F(\boldsymbol{I}_h\boldsymbol{\tau}_2-\boldsymbol{\tau}_2):\llbracket\widetilde{\boldsymbol{u}}_h\rrbracket ds+ \int_{\Omega}(\widetilde{\boldsymbol{u}}_h-\widetilde{\boldsymbol{u}}_h^{\perp})\cdot \widetilde{\boldsymbol{u}}_hdx \\
=&\sum_{F\in\mathcal{F}_h}\int_F(\boldsymbol{I}_h\boldsymbol{\tau}_2-\boldsymbol{\tau}_2):\llbracket\widetilde{\boldsymbol{u}}_h\rrbracket ds+ \|\widetilde{\boldsymbol{u}}_h-\widetilde{\boldsymbol{u}}_h^{\perp}\|_0^2+ \int_{\Omega}(\widetilde{\boldsymbol{u}}_h-\widetilde{\boldsymbol{u}}_h^{\perp})\cdot \widetilde{\boldsymbol{u}}_h^{\perp}dx.
\end{align*}
Together with \eqref{eq:IhErrorEstimate} and \eqref{eq:tau4}, there exists a constant $C_2>0$ such that
\begin{align}
\mathcal B(\widetilde{\boldsymbol{\sigma}}_h, \widetilde{\boldsymbol{u}}_h; \boldsymbol{I}_h\boldsymbol{\tau}_2, 0)=&a(\widetilde{\boldsymbol{\sigma}}_h, \boldsymbol{I}_h\boldsymbol{\tau}_2)+b(\boldsymbol{I}_h\boldsymbol{\tau}_2, \widetilde{\boldsymbol{u}}_h) \notag\\
\geq&\|\widetilde{\boldsymbol{u}}_h-\widetilde{\boldsymbol{u}}_h^{\perp}\|_0^2-C_2\|\widetilde{\boldsymbol{u}}_h-\widetilde{\boldsymbol{u}}_h^{\perp}\|_0(\|\widetilde{\boldsymbol{\sigma}}_h\|_a+\|\widetilde{\boldsymbol{u}}_h\|_c+\|\widetilde{\boldsymbol{u}}_h^{\perp}\|_0) \notag\\
\geq& \frac{1}{2}\|\widetilde{\boldsymbol{u}}_h-\widetilde{\boldsymbol{u}}_h^{\perp}\|_0^2-\frac{3}{2}C_2^2(\|\widetilde{\boldsymbol{\sigma}}_h\|_a^2+\|\widetilde{\boldsymbol{u}}_h\|_c^2+\|\widetilde{\boldsymbol{u}}_h^{\perp}\|_0^2). \label{eq:tau32}
\end{align}
Due to the inverse inequality, there exists a constant $C_3>0$ such that
\[
\|\mathbf{div}\widetilde{\boldsymbol{\sigma}}_h\|_c\leq C_3\|\mathbf{div}\widetilde{\boldsymbol{\sigma}}_h\|_0.
\]
Then we get from the Cauchy-Schwarz inequality
\begin{align}
\mathcal B(\widetilde{\boldsymbol{\sigma}}_h, \widetilde{\boldsymbol{u}}_h; \boldsymbol{0}, -\mathbf{div}\widetilde{\boldsymbol{\sigma}}_h)= & \|\mathbf{div}\widetilde{\boldsymbol{\sigma}}_h\|_0^2 - c(\widetilde{\boldsymbol{u}}_h,  \mathbf{div}\widetilde{\boldsymbol{\sigma}}_h)\geq \|\mathbf{div}\widetilde{\boldsymbol{\sigma}}_h\|_0^2 - \|\widetilde{\boldsymbol{u}}_h\|_c\|\mathbf{div}\widetilde{\boldsymbol{\sigma}}_h\|_c \notag\\
\geq & \|\mathbf{div}\widetilde{\boldsymbol{\sigma}}_h\|_0^2 - C_3\|\widetilde{\boldsymbol{u}}_h\|_c\|\mathbf{div}\widetilde{\boldsymbol{\sigma}}_h\|_0 \geq \frac{1}{2}\|\mathbf{div}\widetilde{\boldsymbol{\sigma}}_h\|_0^2 - \frac{C_3^2}{2}\|\widetilde{\boldsymbol{u}}_h\|_c^2. \label{eq:temp111}
\end{align}

Now take $\boldsymbol{\tau}_h=\widetilde{\boldsymbol{\sigma}}_h+\gamma_1\boldsymbol{\tau}_1+\gamma_2\boldsymbol{I}_h\boldsymbol{\tau}_2$ and $\boldsymbol{v}_h=\widetilde{\boldsymbol{u}}_h-\gamma_3\mathbf{div}\widetilde{\boldsymbol{\sigma}}_h$ where $\gamma_1$, $\gamma_2$ and $\gamma_3$ are three to-be-determined positive constants.
Then we get from \eqref{eq:tau22} and \eqref{eq:tau32}-\eqref{eq:temp111}
\begin{align*}
\mathcal B(\widetilde{\boldsymbol{\sigma}}_h, \widetilde{\boldsymbol{u}}_h; \boldsymbol{\tau}_h, \boldsymbol{v}_h)=&B(\widetilde{\boldsymbol{\sigma}}_h, \widetilde{\boldsymbol{u}}_h; \widetilde{\boldsymbol{\sigma}}_h, \widetilde{\boldsymbol{u}}_h)+\gamma_1B(\widetilde{\boldsymbol{\sigma}}_h, \widetilde{\boldsymbol{u}}_h; \boldsymbol{\tau}_1, 0) \\
&+\gamma_2B(\widetilde{\boldsymbol{\sigma}}_h, \widetilde{\boldsymbol{u}}_h; \boldsymbol{I}_h\boldsymbol{\tau}_2, 0) + \gamma_3B(\widetilde{\boldsymbol{\sigma}}_h, \widetilde{\boldsymbol{u}}_h; \boldsymbol{0}, -\mathbf{div}\widetilde{\boldsymbol{\sigma}}_h) \\
=&\|\widetilde{\boldsymbol{\sigma}}_h\|_a^2+\|\widetilde{\boldsymbol{u}}_h\|_c^2+\gamma_1B(\widetilde{\boldsymbol{\sigma}}_h, \widetilde{\boldsymbol{u}}_h; \boldsymbol{\tau}_1, 0)   \\
&+\gamma_2B(\widetilde{\boldsymbol{\sigma}}_h, \widetilde{\boldsymbol{u}}_h; \boldsymbol{I}_h\boldsymbol{\tau}_2, 0) + \gamma_3B(\widetilde{\boldsymbol{\sigma}}_h, \widetilde{\boldsymbol{u}}_h; \boldsymbol{0}, -\mathbf{div}\widetilde{\boldsymbol{\sigma}}_h) \\
\geq&\left(1-\gamma_1\frac{C_1^2}{2}-\gamma_2\frac{3C_2^2}{2}\right)\|\widetilde{\boldsymbol{\sigma}}_h\|_a^2 + \frac{\gamma_3}{2}\|\mathbf{div}\widetilde{\boldsymbol{\sigma}}_h\|_0^2 +\frac{\gamma_2}{2}\|\widetilde{\boldsymbol{u}}_h-\widetilde{\boldsymbol{u}}_h^{\perp}\|_0^2 \\
&+\left(\frac{\gamma_1}{2}-\gamma_2\frac{3C_2^2}{2}\right)\|\widetilde{\boldsymbol{u}}_h^{\perp}\|_0^2 +\left(1-\gamma_2\frac{3C_2^2}{2}-\gamma_3\frac{C_3^2}{2}\right)\|\widetilde{\boldsymbol{u}}_h\|_c^2 .
\end{align*}
Hence we acquire \eqref{eq:thm1estimate1} by choosing $\gamma_1=\frac{2}{3C_1^2}$, $\gamma_2=\min\{\frac{2}{9C_2^2}, \frac{\gamma_1}{1+3C_2^2}\}$ and $\gamma_3=\frac{2}{3C_3^2}$. And the estimate \eqref{eq:thm1estimate2} follows immediately from the definitions of $\boldsymbol{\tau}_h$ and $\boldsymbol{v}_h$.
\end{proof}

The unique solvability of the stabilized mixed finite element method~\eqref{stabmfem1}-\eqref{stabmfem2} is the immediate result of the inf-sup condition \eqref{eq:infsup}.

Next we show the a priori error analysis for the stabilized mixed finite element method~\eqref{stabmfem1}-\eqref{stabmfem2}.
Subtracting \eqref{stabmfemcompact} from \eqref{mixedformcompact}, we have the following error equation from the definition of $\boldsymbol{Q}_h$
\begin{equation}\label{eq:erroreqncombined}
\mathcal B(\boldsymbol{I}_h\boldsymbol{\sigma}-\boldsymbol{\sigma}_h, \boldsymbol{Q}_h\boldsymbol{u}-\boldsymbol{u}_h; \boldsymbol{\tau}_h, \boldsymbol{v}_h)=a(\boldsymbol{I}_h\boldsymbol{\sigma}-\boldsymbol{\sigma},\boldsymbol{\tau}_h)+b(\boldsymbol{\sigma}-\boldsymbol{I}_h\boldsymbol{\sigma}, \boldsymbol{v}_h)+c(\boldsymbol{Q}_h\boldsymbol{u}, \boldsymbol{v}_h)
\end{equation}
for any $(\boldsymbol{\tau}_h, \boldsymbol{v}_h)\in \boldsymbol{\Sigma}_{k,h}\times \boldsymbol{V}_{k-1,h}$.

\medskip
\begin{theorem}\label{thm:energyerror}
Let $(\boldsymbol{\sigma}, \boldsymbol{u})$ be the exact solution of problem~\eqref{mixedform1}-\eqref{mixedform2} and $(\boldsymbol{\sigma}_h, \boldsymbol{u}_h)$ the discrete solution of the stabilized mixed finite element method~\eqref{stabmfem1}-\eqref{stabmfem2} using $P_k^{\rm div} - P_{k-1}^{-1}$ elements. Assume that $\boldsymbol{\sigma}\in\boldsymbol{H}^{k+1}(\Omega; \mathbb{S})$ and $\boldsymbol{u}\in\boldsymbol{H}^{k}(\Omega; \mathbb{R}^n)$, then
\begin{equation}
\|\boldsymbol{\sigma}-\boldsymbol{\sigma}_h\|_{\boldsymbol{H}(\mathbf{div},\mathcal{A})} + \|\boldsymbol{u}-\boldsymbol{u}_h\|_{0,c}
\lesssim h^k\left (\|\boldsymbol{\sigma}\|_{k+1}+\|\boldsymbol{u}\|_{k}\right ).
\end{equation}
\end{theorem}
\begin{proof}
Set $\widetilde{\boldsymbol{\sigma}}_h=\boldsymbol{I}_h\boldsymbol{\sigma}-\boldsymbol{\sigma}_h$ and $\widetilde{\boldsymbol{u}}_h=\boldsymbol{Q}_h\boldsymbol{u}-\boldsymbol{u}_h$ in Lemma~\ref{infsup}. We have from the error equation \eqref{eq:erroreqncombined}
\begin{align*}
&\|\boldsymbol{I}_h\boldsymbol{\sigma}-\boldsymbol{\sigma}_h\|_{\boldsymbol{H}(\mathbf{div},\mathcal{A})} + \|\boldsymbol{Q}_h\boldsymbol{u}-\boldsymbol{u}_h\|_{0,c} \\
\lesssim & \sup_{(\boldsymbol{\tau}_h, \boldsymbol{v}_h)\in \boldsymbol{\Sigma}_{k,h}\times \boldsymbol{V}_{k-1,h}}\frac{B(\boldsymbol{I}_h\boldsymbol{\sigma}-\boldsymbol{\sigma}_h, \boldsymbol{Q}_h\boldsymbol{u}-\boldsymbol{u}_h; \boldsymbol{\tau}_h, \boldsymbol{v}_h)}{\|\boldsymbol{\tau}_h\|_{\boldsymbol{H}(\mathbf{div},\mathcal{A})} + \|\boldsymbol{v}_h\|_{0,c}} \\
=& \sup_{(\boldsymbol{\tau}_h, \boldsymbol{v}_h)\in \boldsymbol{\Sigma}_{k,h}\times \boldsymbol{V}_{k-1,h}}\frac{a(\boldsymbol{I}_h\boldsymbol{\sigma}-\boldsymbol{\sigma},\boldsymbol{\tau}_h)+b(\boldsymbol{\sigma}-\boldsymbol{I}_h\boldsymbol{\sigma}, \boldsymbol{v}_h)+c(\boldsymbol{Q}_h\boldsymbol{u}-\boldsymbol{u}, \boldsymbol{v}_h)}{\|\boldsymbol{\tau}_h\|_{\boldsymbol{H}(\mathbf{div},\mathcal{A})} + \|\boldsymbol{v}_h\|_{0,c}} \\
\lesssim & \|\boldsymbol{\sigma}-\boldsymbol{I}_h\boldsymbol{\sigma}\|_{\boldsymbol{H}(\mathbf{div},\mathcal{A})} + \|\boldsymbol{u}-\boldsymbol{Q}_h\boldsymbol{u}\|_c.
\end{align*}
Hence we can finish the proof by using the triangle inequality and the interpolation error estimates.
\end{proof}

\section{Two stabilized mixed finite element methods with continuous displacement}

In this section, we will present another class of stabilized mixed finite element methods by a different stabilization mechanism suggested in \cite{BrezziFortinMarini1993}. To pursue a small number of global degrees of freedom, the displacement will be approximated by Lagrange elements.
To be specific, we adopt the following finite element spaces for stress and displacement
\[
\boldsymbol{\Sigma}_{k,h}^{\ast}:=\widetilde{\boldsymbol{\Sigma}}_{k,h} + \boldsymbol{B}_{k+1,h}, \quad \boldsymbol{W}_{k,h}:=\boldsymbol{V}_{k,h}\cap \boldsymbol{H}_0^1(\Omega; \mathbb{R}^n).
\]
Recurring to the stabilization technique in \cite{BrezziFortinMarini1993},
we devise
a stabilized mixed finite element method for linear elasticity as follows:
Find $(\boldsymbol{\sigma}_h, \boldsymbol{u}_h)\in \boldsymbol{\Sigma}_{k,h}^{\ast}\times \boldsymbol{W}_{k,h}$ such that
\begin{align}
a^{\ast}(\boldsymbol{\sigma}_h,\boldsymbol{\tau}_h)+b(\boldsymbol{\tau}_h, \boldsymbol{u}_h)=&-\int_{\Omega}\boldsymbol{f}\cdot\mathbf{div}\boldsymbol{\tau}_h\,{\rm d}x \quad\quad\; \forall\,\boldsymbol{\tau}_h\in\boldsymbol{\Sigma}_{k,h}^{\ast}, \label{continuousmfem1} \\
-b(\boldsymbol{\sigma}_h, \boldsymbol{v}_h)=&\int_{\Omega}\boldsymbol{f}\cdot\boldsymbol{v}_h\,{\rm d}x \quad\quad\quad\quad\quad \forall\,\boldsymbol{v}_h\in \boldsymbol{W}_{k,h}, \label{continuousmfem2}
\end{align}
where $a^{\ast}(\boldsymbol{\sigma}_h,\boldsymbol{\tau}_h):=a(\boldsymbol{\sigma}_h,\boldsymbol{\tau}_h)+\int_{\Omega}\mathbf{div}\boldsymbol{\sigma}_h\cdot\mathbf{div}\boldsymbol{\tau}_h\,{\rm d}x$.
The benefit of this stabilization technique is that the coercivity condition on $\boldsymbol{H}(\mathbf{div}, \Omega; \mathbb{S})$ with norm $\|\cdot\|_{\boldsymbol{H}(\mathbf{div},\mathcal{A})}$ for the bilinear form $a^{\ast}(\cdot,\cdot)$ holds automatically.



We are now in the position to prove the discrete inf-sup condition for the stabilized mixed finite element method~\eqref{continuousmfem1}-\eqref{continuousmfem2}. For this, define an interpolation operator $\boldsymbol{I}_h^{\ast}: \boldsymbol{H}^{1}(\Omega; \mathbb{S})\to\boldsymbol{\Sigma}_{k,h}^{\ast}$ in the following way: for each $\boldsymbol{\tau}\in \boldsymbol{H}^{1}(\Omega; \mathbb{S})$, let
$\boldsymbol{I}_h^{\ast}\boldsymbol{\tau}:=\boldsymbol{I}_h^{SZ}\boldsymbol{\tau}+\boldsymbol{I}_{k+1,h}^b(\boldsymbol{\tau}-\boldsymbol{I}_h^{SZ}\boldsymbol{\tau})$.
As \eqref{eq:Ihprop}-\eqref{eq:IhErrorEstimate}, we have from \eqref{eq:Ihbdef}-\eqref{eq:Ihbbound} and \eqref{eq:IhSZErrorEstimate}
\begin{equation}\label{eq:Ihastprop}
\int_K (\boldsymbol{I}_h^{\ast}\boldsymbol{\tau}-\boldsymbol{\tau}):\boldsymbol{\varsigma}dx =0\quad \forall~\boldsymbol{\varsigma}\in\boldsymbol{P}_{k-1}(K; \mathbb{S}) \textrm{ and } K\in\mathcal{T}_h.
\end{equation}
Similar to Lemma \ref{lm:I1}, we have the following interpolation error estimate.
\begin{lemma}
 Let integers $m,k \geq1$.
We have for any $\boldsymbol{\tau}\in \boldsymbol{H}^{m}(\Omega; \mathbb{S})$
\begin{equation}\label{eq:IhastErrorEstimate}
\sum_{K\in\mathcal{T}_h}h_K^{-2}\|\boldsymbol{\tau}-\boldsymbol{I}_h^{\ast}\boldsymbol{\tau}\|_{0,K}^2+|\boldsymbol{\tau}-\boldsymbol{I}_h^{\ast}\boldsymbol{\tau}|_{1}^2\lesssim h^{2\min\{k,m-1\}}\|\boldsymbol{\tau}\|_{m}^2.
\end{equation}
\end{lemma}

Then using integration by parts and \eqref{eq:Ihastprop}, it holds
\begin{equation}\label{eq:bIhastprop}
b(\boldsymbol{I}_h^{\ast}\boldsymbol{\tau}, \boldsymbol{v}_h)=b(\boldsymbol{\tau}, \boldsymbol{v}_h) \quad \forall~\boldsymbol{\tau}\in \boldsymbol{H}^{1}(\Omega; \mathbb{S}) \textrm{ and } \boldsymbol{v}_h\in\boldsymbol{W}_{k,h}.
\end{equation}

\begin{lemma}
We have the following discrete inf-sup condition
\begin{equation}\label{eq:infsup2}
\|\boldsymbol{v}_h\|_0\lesssim \sup_{\boldsymbol{0}\neq \boldsymbol{\tau}_h\in\boldsymbol{\Sigma}_{k,h}^{\ast}}\frac{b(\boldsymbol{\tau}_h, \boldsymbol{v}_h)}{\|\boldsymbol{\tau}_h\|_{\boldsymbol{H}(\mathbf{div},\mathcal{A})}} \quad \forall~\boldsymbol{v}_h\in\boldsymbol{W}_{k,h}.
\end{equation}
\end{lemma}
\begin{proof}
Let $\boldsymbol{v}_h\in\boldsymbol{W}_{k,h}$, then there exists a $\boldsymbol{\tau}\in\boldsymbol{H}^{1}(\Omega; \mathbb{S})$  such that (cf. \cite{ArnoldWinther2002, ArnoldAwanouWinther2008})
\[
\mathbf{div}\boldsymbol{\tau}=\boldsymbol{v}_h \;\textrm{ and } \;\|\boldsymbol{\tau}\|_1\lesssim \|\boldsymbol{v}_h\|_0.
\]
It follows from \eqref{eq:bIhastprop}
\begin{equation}\label{eq:temp9}
b(\boldsymbol{I}_h^{\ast}\boldsymbol{\tau}, \boldsymbol{v}_h)=b(\boldsymbol{\tau}, \boldsymbol{v}_h)=\|\boldsymbol{v}_h\|_0^2.
\end{equation}
On the other hand, we get from \eqref{eq:IhastErrorEstimate}
\begin{equation}\label{eq:temp10}
\|\boldsymbol{I}_h^{\ast}\boldsymbol{\tau}\|_{\boldsymbol{H}(\mathbf{div},\mathcal{A})}\lesssim \|\boldsymbol{\tau}\|_1\lesssim \|\boldsymbol{v}_h\|_0.
\end{equation}
Hence \eqref{eq:infsup2} is the immediate result of \eqref{eq:temp9}-\eqref{eq:temp10}.
\end{proof}

With previous preparation, we show the a priori error estimate for the $(P_k^0 + B_{k+1}^{\rm div}) - P_k^0$ elements.
\smallskip

\begin{theorem}\label{thm:energyerror2}
Let $(\boldsymbol{\sigma}, \boldsymbol{u})$ be the exact solution of problem~\eqref{mixedform1}-\eqref{mixedform2} and $(\boldsymbol{\sigma}_h, \boldsymbol{u}_h)$ the discrete solution of the stabilized mixed finite element method~\eqref{continuousmfem1}-\eqref{continuousmfem2} using $(P_k^0 + B_{k+1}^{\rm div}) - P_k^0$ element. Assume that $\boldsymbol{\sigma}\in\boldsymbol{H}^{k+1}(\Omega; \mathbb{S})$ and $\boldsymbol{u}\in\boldsymbol{H}^{k+1}(\Omega; \mathbb{R}^n)$, then
\begin{equation*}
\|\boldsymbol{\sigma}-\boldsymbol{\sigma}_h\|_{\boldsymbol{H} (\mathbf{div},\mathcal{A})} + \|\boldsymbol{u}-\boldsymbol{u}_h\|_0\lesssim h^k\left (\|\boldsymbol{\sigma}\|_{k+1}+h\|\boldsymbol{u}\|_{k+1} \right ).
\end{equation*}
\end{theorem}
\begin{proof}
The coercivity of the bilinear form $a^{\ast}(\cdot,\cdot)$ with respect to the norm $\|\cdot\|_{\boldsymbol{H}(\mathbf{div},\mathcal{A})}$ is trivial.
Together with the discrete inf-sup condition \eqref{eq:infsup2}, we obtain the following error estimate by the standard theory of mixed finite element methods (cf. \cite{Brezzi1974, BoffiBrezziFortin2013})
\[
\|\boldsymbol{\sigma}-\boldsymbol{\sigma}_h\|_{\boldsymbol{H}(\mathbf{div},\mathcal{A})} + \|\boldsymbol{u}-\boldsymbol{u}_h\|_0\lesssim \inf_{\boldsymbol{\tau}_h\in\boldsymbol{\Sigma}_{k,h}^{\ast}, \boldsymbol{v}_h\in \boldsymbol{W}_{k,h}}\left(\|\boldsymbol{\sigma}-\boldsymbol{\tau}_h\|_{\boldsymbol{H}(\mathbf{div},\mathcal{A})} + \|\boldsymbol{u}-\boldsymbol{v}_h\|_0\right).
\]
Choose $\boldsymbol{\tau}_h=\boldsymbol{I}_h^{\ast}\boldsymbol{\sigma}$ and $\boldsymbol{v}_h=\boldsymbol{I}_h^{SZ}\boldsymbol{u}$. Then we can finish the proof by combining the last inequality, \eqref{eq:IhastErrorEstimate} and \eqref{eq:IhSZErrorEstimate}.
\end{proof}


To achieve the optimal convergence rate of $\|\boldsymbol{u}-\boldsymbol{u}_h\|_0$, we can further enrich the stress finite element space
 $\boldsymbol{\Sigma}_{k,h}^{\ast}$ to $\boldsymbol{\Sigma}_{k+1,h}$.
The resulting mixed finite element method is: Find $(\boldsymbol{\sigma}_h, \boldsymbol{u}_h)\in \boldsymbol{\Sigma}_{k+1,h}\times \boldsymbol{W}_{k,h}$ such that
\begin{align}
a^{\ast}(\boldsymbol{\sigma}_h,\boldsymbol{\tau}_h)+b(\boldsymbol{\tau}_h, \boldsymbol{u}_h)=&-\int_{\Omega}\boldsymbol{f}\cdot\mathbf{div}\boldsymbol{\tau}_h\,{\rm d}x \quad\quad\; \forall\,\boldsymbol{\tau}_h\in\boldsymbol{\Sigma}_{k+1,h}, \label{continuousfullmfem1} \\
-b(\boldsymbol{\sigma}_h, \boldsymbol{v}_h)=&\int_{\Omega}\boldsymbol{f}\cdot\boldsymbol{v}_h\,{\rm d}x \quad\quad\quad\quad\quad \forall\,\boldsymbol{v}_h\in \boldsymbol{W}_{k,h}, \label{continuousfullmfem2}
\end{align}

\smallskip
\begin{corollary}
Let $(\boldsymbol{\sigma}, \boldsymbol{u})$ be the exact solution of problem~\eqref{mixedform1}-\eqref{mixedform2} and $(\boldsymbol{\sigma}_h, \boldsymbol{u}_h)$ the discrete solution of the stabilized mixed finite element method~\eqref{continuousfullmfem1}-\eqref{continuousfullmfem2} using $P_{k+1}^{\rm div} - P_k^0$ element.
 Then under the assumption of $\boldsymbol{\sigma}\in\boldsymbol{H}^{k+2}(\Omega; \mathbb{S})$ and $\boldsymbol{u}\in\boldsymbol{H}^{k+1}(\Omega; \mathbb{R}^n)$, it follows
 \begin{equation}\label{eq:errorestimate4continuousfullmfem}
\|\boldsymbol{\sigma}-\boldsymbol{\sigma}_h\|_{\boldsymbol{H}(\mathbf{div},\mathcal{A})} + \|\boldsymbol{u}-\boldsymbol{u}_h\|_0\lesssim h^{k+1}(\|\boldsymbol{\sigma}\|_{k+2}+\|\boldsymbol{u}\|_{k+1}).
\end{equation}
\end{corollary}

\begin{remark}
The finite element pair $\boldsymbol{\Sigma}_{k+1,h}\times \boldsymbol{W}_{k,h}$ is just the Hood-Taylor element in \cite{TaylorHood1973,Boffi1994,Boffi1997} augmented by the elementwise $\boldsymbol{H}(\mathbf{div})$ bubble function space $\boldsymbol{B}_{k+1,h}$.
Hence we give a positive answer to the question in \cite[Example~3.3]{BrezziFortinMarini1993} that whether the Hood-Taylor element is stable for the linear elasticity.
\end{remark}

\begin{remark}
In order to remain the right hand side of \eqref{continuousfullmfem1} to be zero as in \eqref{mixedform1} and \eqref{stabmfem1},
we can use the following stabilized mixed finite element method: Find $(\boldsymbol{\sigma}_h, \boldsymbol{u}_h)\in \boldsymbol{\Sigma}_{k+1,h}\times \boldsymbol{W}_{k,h}$ such that
\begin{align*}
a^{\circ}(\boldsymbol{\sigma}_h,\boldsymbol{\tau}_h)+b(\boldsymbol{\tau}_h, \boldsymbol{u}_h)=&0 \quad\quad\quad\quad\quad\quad\; \forall\,\boldsymbol{\tau}_h\in\boldsymbol{\Sigma}_{k+1,h}, \\
-b(\boldsymbol{\sigma}_h, \boldsymbol{v}_h)=&\int_{\Omega}\boldsymbol{f}\cdot\boldsymbol{v}_h\,{\rm d}x \quad\quad \forall\,\boldsymbol{v}_h\in \boldsymbol{W}_{k,h}, 
\end{align*}
where
$
a^{\circ}(\boldsymbol{\sigma},\boldsymbol{\tau})=\int_{\Omega}\mathcal{A}\boldsymbol{\sigma}:\boldsymbol{\tau}\,{\rm d}x+\sum\limits_{F\in\mathcal{F}_h}h_F\int_{F}[\mathbf{div}\boldsymbol{\sigma}]\cdot[\mathbf{div}\boldsymbol{\tau}]\,{\rm d}s$.
For any $\boldsymbol{\tau}_h\in \boldsymbol{\Sigma}_{k+1,h}$, define norm $\interleave \boldsymbol{\tau}_h\interleave^2:=\|\boldsymbol{\tau}_h\|_{\boldsymbol{H}(\mathbf{div},\mathcal{A})}^2+\sum\limits_{F\in\mathcal{F}_h}h_F\|[\mathbf{div}\boldsymbol{\tau}_h]\|_{0,F}^2$.
It can be shown that $a^{\circ}(\cdot,\cdot)$ is coercive on the kernel space
\[
\boldsymbol{K}_h:=\{\boldsymbol{\tau}_h\in\boldsymbol{\Sigma}_{k+1,h}: b(\boldsymbol{\tau}_h, \boldsymbol{v}_h)=0\quad\forall~\boldsymbol{v}_h\in \boldsymbol{W}_{k,h}\},
\]
i.e.
\[
\interleave\boldsymbol{\tau}_h\interleave^2\lesssim a^{\circ}(\boldsymbol{\tau}_h,\boldsymbol{\tau}_h) \quad \forall~\boldsymbol{\tau}_h\in \boldsymbol{K}_h.
\]
And the discrete inf-sup condition can be derived from \eqref{eq:infsup2} and the inverse inequality.
\end{remark}


\section{Numerical Results}

In this section, we will report some numerical results to assess the accuracy and behavior of the stabilized mixed finite element methods developed in Sections~3-4. Let $\lambda=0.3$ and $\mu=0.35$. We use the uniform triangulation $\mathcal{T}_h$ of $\Omega$.

First we test our stabilized mixed finite element methods for the pure displacement problem on the square $\Omega=(-1,1)^2$ in 2D.
Take
\begin{align*}
\boldsymbol{f}(x_1,x_2)=&\left(\begin{array}{l}
-8(x_1 + x_2)\left((3x_1x_2-2)(x_1^2+x_2^2)+5(x_1x_2-1)^2-2x_1^2x_2^2\right) \\
-8(x_1 - x_2)\left((3x_1x_2+2)(x_1^2+x_2^2)-5(x_1x_2+1)^2+2x_1^2x_2^2\right)
\end{array}\right).
\end{align*}
It can be verified that the exact displacement of problem~\eqref{mixedform1}-\eqref{mixedform2} is
\[
\boldsymbol{u}(x_1,x_2)=\frac{80}{7}\left(\begin{array}{l}
-x_2(1-x_2^2)(1-x_1^2)^2 \\
x_1(1-x_1^2)(1-x_2^2)^2
\end{array}\right)-4\left(\begin{array}{l}
x_1(1-x_1^2)(1-x_2^2)^2 \\
x_2(1-x_2^2)(1-x_1^2)^2
\end{array}\right).
\]
And the exact stress can be computed by $\boldsymbol{\sigma}=2\mu\boldsymbol{\varepsilon}(\boldsymbol{u})+\lambda(\textrm{tr}\boldsymbol{\varepsilon}(\boldsymbol{u}))\boldsymbol{\delta}$.

The element diagram in Fig.~\ref{fig: dofsStressk2n2} is mnemonic of the local degrees of freedom of $\boldsymbol{\Sigma}_{2,h}$ in 2D.
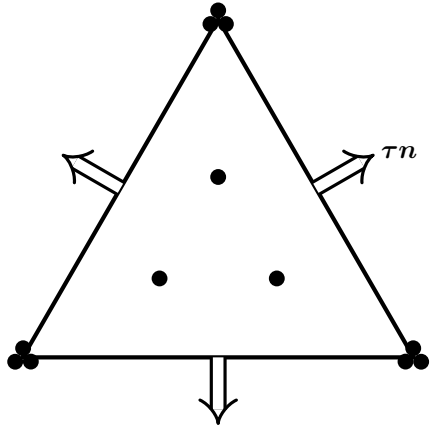
\begin{figure}
  \centering
\begin{tikzpicture}[scale=3]
\path (0,0) coordinate (o);
\path (90:1cm) coordinate (p1);
\path (210:1cm) coordinate (p2);
\path (330:1cm) coordinate (p3);
\path (30:0.5cm) coordinate (m11);
\path (30:0.8cm) coordinate (m12);
\path (150:0.5cm) coordinate (m21);
\path (150:0.8cm) coordinate (m22);
\path (270:0.5cm) coordinate (m31);
\path (270:0.8cm) coordinate (m32);
\draw[ultra thick] (p1) -- (p2) -- (p3) -- cycle;
\fill (90:1cm+1.1pt) circle (1pt);
\fill (92:1cm-0.6pt) circle (1pt);
\fill (88:1cm-0.6pt) circle (1pt);

\fill (210:1cm+1.1pt) circle (1pt);
\fill (212:1cm-0.6pt) circle (1pt);
\fill (208:1cm-0.6pt) circle (1pt);

\fill (330:1cm+1.1pt) circle (1pt);
\fill (332:1cm-0.6pt) circle (1pt);
\fill (328:1cm-0.6pt) circle (1pt);

\draw[double distance=4pt, very thick,-implies] (m11) -- (m12);
\draw[double distance=4pt, very thick,-implies] (m21) -- (m22);
\draw[double distance=4pt, very thick,-implies] (m31) -- (m32);
\path (27:0.9cm) node {$\boldsymbol{\tau}\boldsymbol{n}$};

\fill (90:0.3cm) circle (1pt);
\fill (210:0.3cm) circle (1pt);
\fill (330:0.3cm) circle (1pt);
\end{tikzpicture}
  \caption{Element diagram for $\boldsymbol{\Sigma}_{2,h}$ in 2D}\label{fig: dofsStressk2n2}
\end{figure}
For the stabilized $P_{k}^{\rm div} - P_{k-1}^{-1}$ element, numerical errors $\|\boldsymbol{\sigma}-\boldsymbol{\sigma}_h\|_{\boldsymbol{H}(\mathbf{div},\mathcal{A})}$, $\|u_h\|_c$ and $\|\boldsymbol{u}-\boldsymbol{u}_h\|_0$ with respect to $h$ for $k=1, 2$ are shown in Tables~\ref{table:errorstabmixedk1}-\ref{table:errorstabmixedk2}, from which we can see that all the three errors achieve optimal convergence rates $O(h^k)$ numerically.
These results agree with the theoretical result in Theorem~\ref{thm:energyerror}.
Numerical results for the stabilized mixed finite element method~\eqref{continuousmfem1}-\eqref{continuousmfem2} with $k=1$ are listed in Table~\ref{table:errormodimixedk1}.
We find that the convergence rate of $\|\boldsymbol{\sigma}-\boldsymbol{\sigma}_h\|_{\boldsymbol{H}(\mathbf{div},\mathcal{A})}$ is $O(h)$, which coincides with
Theorem~\ref{thm:energyerror2}. It deserves to be mentioned that the convergence rate of $\|\boldsymbol{u}-\boldsymbol{u}_h\|_0$ in Table~\ref{table:errormodimixedk1} is higher than the theoretical result in Theorem~\ref{thm:energyerror2}, but still suboptimal.
Numerical results for the stabilized mixed finite element method~\eqref{continuousfullmfem1}-\eqref{continuousfullmfem2} with $k=1$ are given in Tables~\ref{table:errormodimixedfullk1}. It can be observed that both the convergence rates of $\|\boldsymbol{\sigma}-\boldsymbol{\sigma}_h\|_{\boldsymbol{H}(\mathbf{div},\mathcal{A})}$ and $\|\boldsymbol{u}-\boldsymbol{u}_h\|_0$ are $O(h^2)$,
as indicated by \eqref{eq:errorestimate4continuousfullmfem}.

\begin{table}[!h]
\tabcolsep 5pt \caption{Numerical errors for the stabilized $P_{1}^{0} - P_0^{-1}$ element in 2D.}\label{table:errorstabmixedk1} \vspace*{-12pt}
\begin{center}
\def\temptablewidth{0.91\textwidth}
{\rule{\temptablewidth}{1pt}}
\begin{tabular*}{\temptablewidth}{@{\extracolsep{\fill}}|c|c|c|c|c|c|c|}
\hline \rule{0em}{1.3em} $h$ & $\|\boldsymbol{\sigma}-\boldsymbol{\sigma}_h\|_{\boldsymbol{H}(\mathbf{div},\mathcal{A})}$ & order & $\|\boldsymbol{u}_h\|_c$ & order & $\|\boldsymbol{u}-\boldsymbol{u}_h\|_0$ & order \\ 
\hline $2^{-1}$ & 1.9436E+01 & $-$ & 5.7136E+00 & $-$ & 2.8981E+00 & $-$ \\
\hline $2^{-2}$ & 1.0703E+01 & 0.86 & 3.7894E+00 & 0.59 & 1.6073E+00 & 0.85 \\
\hline $2^{-3}$ & 5.7982E+00 & 0.88 & 2.1600E+00 & 0.81 & 8.3356E-01 & 0.95 \\
\hline $2^{-4}$ & 3.0580E+00 & 0.92 & 1.1484E+00 & 0.91 & 4.2521E-01 & 0.97 \\
\hline $2^{-5}$ & 1.5780E+00 & 0.95 & 5.9220E-01 & 0.96 & 2.1527E-01 & 0.98 \\
\hline $2^{-6}$ & 8.0346E-01 & 0.97 & 3.0101E-01 & 0.98 & 1.0848E-01 & 0.99 \\
\hline $2^{-7}$ & 4.0590E-01 & 0.99 & 1.5187E-01 & 0.99 & 5.4494E-02 & 0.99 \\
\hline
       \end{tabular*}
       {\rule{\temptablewidth}{1pt}}
       \end{center}
\end{table}
\begin{table}[!h]
\tabcolsep 5pt \caption{Numerical errors for the stabilized $P_{2}^{\rm div} - P_1^{-1}$ element in 2D.}\label{table:errorstabmixedk2} \vspace*{-12pt}
\begin{center}
\def\temptablewidth{0.91\textwidth}
{\rule{\temptablewidth}{1pt}}
\begin{tabular*}{\temptablewidth}{@{\extracolsep{\fill}}|c|c|c|c|c|c|c|}
\hline \rule{0em}{1.3em} $h$ & $\|\boldsymbol{\sigma}-\boldsymbol{\sigma}_h\|_{\boldsymbol{H}(\mathbf{div},\mathcal{A})}$ & order & $\|\boldsymbol{u}_h\|_c$ & order & $\|\boldsymbol{u}-\boldsymbol{u}_h\|_0$ & order \\ 
\hline $1$ & 1.1868E+01 & $-$ & 5.0478E+00 & $-$ & 2.4374E+00 & $-$ \\
\hline $2^{-1}$ & 4.6400E+00 & 1.35 & 1.7436E+00 & 1.53 & 7.1254E-01 & 1.77 \\
\hline $2^{-2}$ & 1.4841E+00 & 1.64 & 4.6132E-01 & 1.92 & 1.8285E-01 & 1.96 \\
\hline $2^{-3}$ & 4.2227E-01 & 1.81 & 1.1783E-01 & 1.97 & 4.6102E-02 & 1.99 \\
\hline $2^{-4}$ & 1.1120E-01 & 1.92 & 2.9546E-02 & 2.00 & 1.1556E-02 & 2.00 \\
\hline $2^{-5}$ & 2.8378E-02 & 1.97 & 7.3651E-03 & 2.00 & 2.8912E-03 & 2.00 \\
\hline $2^{-6}$ & 7.1562E-03 & 1.99 & 1.8358E-03 & 2.00 & 7.2294E-04 & 2.00 \\
\hline
       \end{tabular*}
       {\rule{\temptablewidth}{1pt}}
       \end{center}
\end{table}
\begin{table}[!h]
\tabcolsep 5pt \caption{Numerical errors for the stabilized $(P_1^0 + B_{2}^{\rm div}) - P_1^0$ element in 2D.}\label{table:errormodimixedk1} \vspace*{-12pt}
\begin{center}
\def\temptablewidth{0.7\textwidth}
{\rule{\temptablewidth}{1pt}}
\begin{tabular*}{\temptablewidth}{@{\extracolsep{\fill}}|c|c|c|c|c|}
\hline \rule{0em}{1.3em} $h$ & $\|\boldsymbol{\sigma}-\boldsymbol{\sigma}_h\|_{\boldsymbol{H}(\mathbf{div},\mathcal{A})}$ & order & $\|\boldsymbol{u}-\boldsymbol{u}_h\|_0$ & order \\ 
\hline $2^{-1}$ & 1.3570E+01 & $-$ & 5.9057E+00 & $-$ \\
\hline $2^{-2}$ & 7.5576E+00 & 0.84 & 2.1407E+00 & 1.46 \\
\hline $2^{-3}$ & 4.1592E+00 & 0.86 & 6.2487E-01 & 1.78 \\
\hline $2^{-4}$ & 2.2977E+00 & 0.86 & 1.9626E-01 & 1.67 \\
\hline $2^{-5}$ & 1.2391E+00 & 0.89 & 6.2250E-02 & 1.66 \\
\hline $2^{-6}$ & 6.4969E-01 & 0.93 & 1.9087E-02 & 1.71 \\
\hline $2^{-7}$ & 3.3399E-01 & 0.96 & 5.7719E-03 & 1.73 \\
\hline
       \end{tabular*}
       {\rule{\temptablewidth}{1pt}}
       \end{center}
\end{table}
\begin{table}[!h]
\tabcolsep 5pt \caption{Numerical errors for the stabilized $P_2^{\rm div} - P_1^0$ element in 2D.}\label{table:errormodimixedfullk1} \vspace*{-12pt}
\begin{center}
\def\temptablewidth{0.7\textwidth}
{\rule{\temptablewidth}{1pt}}
\begin{tabular*}{\temptablewidth}{@{\extracolsep{\fill}}|c|c|c|c|c|}
\hline \rule{0em}{1.3em} $h$ & $\|\boldsymbol{\sigma}-\boldsymbol{\sigma}_h\|_{\boldsymbol{H}(\mathbf{div},\mathcal{A})}$ & order & $\|\boldsymbol{u}-\boldsymbol{u}_h\|_0$ & order \\ 
\hline $1$ & 1.0966E+01 & $-$ & 6.0260E+00 & $-$ \\
\hline $2^{-1}$ & 3.5092E+00 & 1.64 & 1.5579E+00 & 1.95 \\
\hline $2^{-2}$ & 9.0380E-01 & 1.96 & 3.3148E-01 & 2.23 \\
\hline $2^{-3}$ & 2.2504E-01 & 2.01 & 7.2219E-02 & 2.20 \\
\hline $2^{-4}$ & 5.5922E-02 & 2.01 & 1.6506E-02 & 2.13 \\
\hline $2^{-5}$ & 1.3981E-02 & 2.00 & 4.1182E-03 & 2.00 \\
\hline $2^{-6}$ & 3.4746E-03 & 2.01 & 9.5159E-04 & 2.11 \\
\hline
       \end{tabular*}
       {\rule{\temptablewidth}{1pt}}
       \end{center}
\end{table}

Next we take into account the pure displacement problem on the unit cube $\Omega=(0,1)^3$ in 3D.
The exact solution is given by
\[
\boldsymbol{u}(x_1,x_2,x_3)=\left(\begin{array}{l}
2^4 \\
2^5 \\
2^6
\end{array}\right)x_1(1-x_1)x_2(1-x_2)x_3(1-x_3).
\]
Then the exact stress $\boldsymbol{\sigma}$ and the load function $\boldsymbol{f}$ are derived from \eqref{mixedform1}-\eqref{mixedform2}.
From Tables~\ref{table:errorstabmixedk13D}-\ref{table:errorstabmixedk23D}, it is easy to see that all the convergence rates of the errors $\|\boldsymbol{\sigma}-\boldsymbol{\sigma}_h\|_{\boldsymbol{H}(\mathbf{div},\mathcal{A})}$, $\|\boldsymbol u_h\|_c$ and $\|\boldsymbol{u}-\boldsymbol{u}_h\|_0$
for the stabilized $P_k^{\rm div} - P_{k-1}^{-1}$ with $k=1, 2$ are optimal, i.e. $O(h^k)$ assured by Theorem~\ref{thm:energyerror}.
The numerical convergence rates of the errors $\|\boldsymbol{\sigma}-\boldsymbol{\sigma}_h\|_{\boldsymbol{H}(\mathbf{div},\mathcal{A})}$ and $\|\boldsymbol{u}-\boldsymbol{u}_h\|_0$ for the stabilized $P_{k+1}^{\rm div}-P_{k}^0$ element with $k=1, 2$ are
presented in Tables~\ref{table:errormodimixedfullk13D}-\ref{table:errormodimixedfullk23D}. The numerical results in these two tables confirm the optimal rate of convergence result\eqref{eq:errorestimate4continuousfullmfem}.
\begin{table}[!h]
\tabcolsep 5pt \caption{Numerical errors for the stabilized $P_{1}^{0} - P_0^{-1}$ element in 3D.}\label{table:errorstabmixedk13D} \vspace*{-12pt}
\begin{center}
\def\temptablewidth{0.91\textwidth}
{\rule{\temptablewidth}{1pt}}
\begin{tabular*}{\temptablewidth}{@{\extracolsep{\fill}}|c|c|c|c|c|c|c|}
\hline \rule{0em}{1.3em} $h$ & $\|\boldsymbol{\sigma}-\boldsymbol{\sigma}_h\|_{\boldsymbol{H}(\mathbf{div},\mathcal{A})}$ & order & $\|\boldsymbol{u}_h\|_c$ & order & $\|\boldsymbol{u}-\boldsymbol{u}_h\|_0$ & order \\ 
\hline $2^{-1}$ & 4.1723E+00 & $-$ & 4.0747E-01 & $-$ & 2.4720E-01 & $-$ \\
\hline $2^{-2}$ & 2.3595E+00 & 0.82 & 3.5554E-01 & 0.20 & 1.7403E-01 & 0.51 \\
\hline $2^{-3}$ & 1.2849E+00 & 0.88 & 2.5527E-01 & 0.48 & 1.1168E-01 & 0.64 \\
\hline $2^{-4}$ & 6.8023E-01 & 0.92 & 1.5243E-01 & 0.74 & 6.3889E-02 & 0.81 \\
\hline $2^{-5}$ & 3.5167E-01 & 0.95 & 8.3310E-02 & 0.87 & 3.4309E-02 & 0.90 \\
\hline
       \end{tabular*}
       {\rule{\temptablewidth}{1pt}}
       \end{center}
\end{table}
\begin{table}[!h]
\tabcolsep 5pt \caption{Numerical errors for the stabilized $P_{2}^{\rm div} - P_1^{-1}$ element in 3D.}\label{table:errorstabmixedk23D} \vspace*{-12pt}
\begin{center}
\def\temptablewidth{0.91\textwidth}
{\rule{\temptablewidth}{1pt}}
\begin{tabular*}{\temptablewidth}{@{\extracolsep{\fill}}|c|c|c|c|c|c|c|}
\hline \rule{0em}{1.3em} $h$ & $\|\boldsymbol{\sigma}-\boldsymbol{\sigma}_h\|_{\boldsymbol{H}(\mathbf{div},\mathcal{A})}$ & order & $\|\boldsymbol{u}_h\|_c$ & order & $\|\boldsymbol{u}-\boldsymbol{u}_h\|_0$ & order \\ 
\hline $2^{-1}$ & 1.4440E+00 & $-$ & 1.7738E-01 & $-$ & 8.3035E-02 & $-$ \\
\hline $2^{-2}$ & 3.8864E-01 & 1.89 & 5.2337E-02 & 1.76 & 2.2979E-02 & 1.85 \\
\hline $2^{-3}$ & 9.9734E-02 & 1.96 & 1.3657E-02 & 1.94 & 5.9084E-03 & 1.96 \\
\hline $2^{-4}$ & 2.5160E-02 & 1.99 & 3.4507E-03 & 1.98 & 1.4873E-03 & 1.99 \\
\hline
       \end{tabular*}
       {\rule{\temptablewidth}{1pt}}
       \end{center}
\end{table}
\begin{table}[!h]
\tabcolsep 5pt \caption{Numerical errors for the stabilized $P_2^{\rm div} - P_1^0$ element in 3D.}\label{table:errormodimixedfullk13D} \vspace*{-12pt}
\begin{center}
\def\temptablewidth{0.7\textwidth}
{\rule{\temptablewidth}{1pt}}
\begin{tabular*}{\temptablewidth}{@{\extracolsep{\fill}}|c|c|c|c|c|}
\hline \rule{0em}{1.3em} $h$ & $\|\boldsymbol{\sigma}-\boldsymbol{\sigma}_h\|_{\boldsymbol{H}(\mathbf{div},\mathcal{A})}$ & order & $\|\boldsymbol{u}-\boldsymbol{u}_h\|_0$ & order \\ 
\hline $2^{-1}$ & 1.4391E+00 & $-$ & 2.3509E-01 & $-$ \\
\hline $2^{-2}$ & 3.8148E-01 & 1.92 & 5.4959E-02 & 2.10 \\
\hline $2^{-3}$ & 9.6524E-02 & 1.98 & 1.1730E-02 & 2.23 \\
\hline $2^{-4}$ & 2.4182E-02 & 2.00 & 2.6368E-03 & 2.15 \\
\hline
       \end{tabular*}
       {\rule{\temptablewidth}{1pt}}
       \end{center}
\end{table}
\begin{table}[!h]
\tabcolsep 5pt \caption{Numerical errors for the stabilized $P_3^{\rm div} - P_2^0$ element in 3D.}\label{table:errormodimixedfullk23D} \vspace*{-12pt}
\begin{center}
\def\temptablewidth{0.7\textwidth}
{\rule{\temptablewidth}{1pt}}
\begin{tabular*}{\temptablewidth}{@{\extracolsep{\fill}}|c|c|c|c|c|}
\hline \rule{0em}{1.3em} $h$ & $\|\boldsymbol{\sigma}-\boldsymbol{\sigma}_h\|_{\boldsymbol{H}(\mathbf{div},\mathcal{A})}$ & order & $\|\boldsymbol{u}-\boldsymbol{u}_h\|_0$ & order \\ 
\hline $2^{-1}$ & 2.7531E-01 & $-$ & 3.9149E-02 & $-$ \\
\hline $2^{-2}$ & 3.7035E-02 & 2.89 & 5.7416E-03 & 2.77 \\
\hline $2^{-3}$ & 4.7120E-03 & 2.97 & 7.8312E-04 & 2.87 \\
\hline
       \end{tabular*}
       {\rule{\temptablewidth}{1pt}}
       \end{center}
\end{table}


\end{document}